\documentclass[10pt]{amsart}

\usepackage{amsmath}
  \usepackage{paralist}
  \usepackage{graphics} 
  \usepackage{epsfig} 
\usepackage{graphicx}  \usepackage{epstopdf}

 \usepackage[colorlinks=true]{hyperref}
\hypersetup{urlcolor=blue, citecolor=red}

\newtheorem{theorem}{Theorem}[section]
\newtheorem{corollary}[theorem]{Corollary}

\newtheorem*{main*}{Main Theorem}
\newtheorem{lemma}[theorem]{Lemma}
\newtheorem{proposition}[theorem]{Proposition}

\theoremstyle{definition}
\newtheorem{definition}[theorem]{Definition}

\newtheorem{example}[theorem]{Example}

\title[Running heading with forty characters or less]
      {On non-dense orbits of certain non-algebraic dynamical systems}

\author[first-name1 last-name1 and first-name2 last-name2]{Weisheng Wu}

\subjclass{}
 \keywords{Non-dense orbits; Schmidt games; Hausdorff dimension; Anosov diffeomorphism; Expanding endomorphism}

\address{School of Mathematical Sciences, Peking University, Beijing, 100871, China}
 \email{wuweisheng@math.pku.edu.cn}

\begin{document}
\maketitle
\markboth{On non-dense orbits of certain non-algebraic dynamical systems}
{On non-dense orbits of certain non-algebraic dynamical systems}
\renewcommand{\sectionmark}[1]{}

\begin{abstract}
In this paper, we manage to apply Schmidt games to certain non-algebraic dynamical systems. More precisely, we show that the set of points with non-dense forward orbit under a $C^2$-Anosov diffeomorphism with conformality on unstable manifolds is a winning set for Schmidt games. It is also proved that for a $C^{1+\theta}$-expanding endomorphism the set of points with non-dense forward orbit is a winning set for certain variants of Schmidt games.
\end{abstract}

\section{Introduction}
Let $M$ be an $n$-dimensional smooth and compact Riemannian manifold without boundary, and $f: M \rightarrow M$ be a $C^2$-diffeomorphism or endomorphism. In this paper, we will study the set of points with non-dense forward orbit under $f$. We call it the non-dense set of $f$ and denote it by $ND(f)$. If $f$ preserves an ergodic measure of full support, then $ND(f)$ has measure zero. But in many dynamical systems with hyperbolic behavior, this exceptional set is large in Hausdorff dimension. It was proved in \cite{U} that for a $C^2$-expanding endomorphism or a transitive $C^2$-Anosov diffeomorphism the non-dense set has full Hausdorff dimension equal to the dimension of the underlying manifold.

An effective tool for proving full Hausdorff dimension is Schmidt games, which were first introduced by W. M. Schmidt in \cite{S}. A winning set for such games is large in the following sense: it is dense in the metric space, and its intersection with any nonempty open subset has full Hausdorff dimension when the metric space is $\mathbb{R}^n$ or a manifold. Moreover, the winning property is stable with respect to countable intersections. Schmidt proved in \cite{S} that the set of badly approximable numbers is winning for Schmidt games and hence has full Hausdorff dimension $1$. There are many applications of Schmidt games in homogeneous dynamics due to the well known connection between Diophantine approximation and bounded orbits of nonquasiunipotent flows on homogeneous spaces (see \cite{S2}, \cite{D1}, \cite{D2}, \cite{AL}, \cite{CCM}, \cite{KM}, \cite{Dol}, \cite{KW1}, \cite{KW2}, \cite{DS}, \cite{AN} and many others). Motivated by this, a remarkable result was proved that for any toral endomorphism on $\mathbb{T}^n$ the non-dense set is winning for Schmidt games (cf. \cite{D} and \cite{BFK}). So far, most of known dynamical systems that have winning non-dense set are of algebraic character.

It is natural to ask which non-algebraic dynamical systems have winning non-dense sets. In \cite{U}, it is mentioned that the theory of Schmidt games is specific for the algebraic case and rather inapplicable to general Anosov diffeomorphisms. Recently, J. Tseng proved that for a $C^2$-expanding endomorphism on the circle the non-dense set is a winning set for Schmidt games and asked a question of whether there are non-algebraic dynamical systems with winning non-dense set in dimensions greater than one (cf. \cite{tseng}). Later he answered this question and proved that certain Anosov diffeomorphism on the $2$-torus has a winning non-dense set, by using the $C^1$ conjugacy between such system and a linear hyperbolic automorphism on the $2$-torus (cf. \cite{tseng1}). The problem of which non-algebraic dynamical systems in dimensions greater than one have non-dense sets winning for Schmidt games remains widely open.

\subsection{Anosov diffeomorphisms with conformality on unstable manifolds}
In this paper, we provide a class of non-algebraic dynamical systems in dimensions greater than one whose non-dense sets are winning for Schmidt games. Let $f: M\to M$ be a diffeomorphism (or an endomorphism). Fix $y\in M$ and define
\begin{equation*}
E(f, y):= \{ z\in M: y\notin \overline{\{f^k(z), k \in \mathbb{N}\}}\}.
\end{equation*}
By definition, any point in $E(f, y)$ has a non-dense forward orbit in $M$, namely $E(f,y) \subset ND(f)$. Our main result is the following:
\begin{theorem}\label{main1}
Let $f: M\to M$ be a $C^2$-Anosov diffeomorphism. Suppose that $f$ is conformal on unstable manifolds, i.e., for each $x\in M$, the derivative map $D_xf|_{E_x^u}$ is a scalar multiple of an isometry. Then $E(f,y)$ is a winning set for Schmidt games played on $M$.
\end{theorem}
One will see that conformality of $f$ in $E^u$ plays a crucial role in Theorem \ref{main1}. Indeed, we will develop the geometric method in \cite{Wu} which was inspired by \cite{BFK}, to give a winning strategy for Schmidt games. The existence of invariant foliation $W^s$ guarantees that the preimages (under $f$) of certain open set that we try to avoid are "uniformly displayed" inside the balls in Schmidt games. Thus we can avoid these preimages by avoiding them in expanding direction $E^u$, which in turn is possible due to conformality by the method in \cite{Wu}.

\subsection{Expanding endomorphisms}
We also study the non-dense set of a $C^{1+\theta}$-expanding endomorphism $f$ which in general is non-algebraic. In dimensions greater than one, we don't know whether $ND(f)$ is winning for Schmidt games in general. But utilizing certain variants of Schmidt games, we can obtain some nice properties of $ND(f)$. The first variant of Schmidt games we consider here is absolute games (cf. \cite{Mc1}, \cite{BFKRW}, \cite{FSU}, etc.). The following theorem generalizes Tseng's result for $C^2$-expanding endomorphisms on the circle (cf. \cite{tseng}):
\begin{theorem}\label{main2}
Let $f: M\to M$ be a $C^{1+\theta}$-expanding endomorphism. Suppose that $f$ is conformal, i.e., for each $x\in M$, the derivative map $D_xf$ is a scalar multiple of an isometry. Then $E(f,y)$ is absolute winning.
\end{theorem}
Absolute winning implies winning for Schmidt games. Many properties are enjoyed by absolute winning sets that are not true for winning sets. For $Y \subset M$, we define
\begin{equation*}
E(f, Y):= \{ z\in M: Y \cap \overline{\{f^k(z), k \in \mathbb{N}\}}=\emptyset\}.
\end{equation*}
The following is an immediate consequence of absolute winning property.
\begin{corollary}\label{coro2}
Let  $\{f_i\}_{i=1}^\infty$ be a countable set of $C^{1+\theta}$-expanding endomorphisms which are conformal on $M$, and $Y$ be a countable subset of $M$. Then $\cap_{i=1}^\infty E(f_i,Y)$ is absolute winning.
\end{corollary}
In Question $1$ at the end of \cite{tseng}, J. Tseng asked the question that whether $E(f,y)$ is $\alpha$-winning for Schmidt games, for some $\alpha$ independent of the choice of Markov partition and of $f$ itself. In particular, Theorem \ref{main2} answers this question, and implies that $E(f,y)$ is in fact $1/2$-winning (see Proposition \ref{winningproperty} below). The proof of Theorem \ref{main2} is a nontrivial improvement of the geometric method in \cite{Wu} to the setting of conformal $C^{1+\theta}$-expanding endomorphisms.

The second variant of Schmidt games which we would like to consider is called modified Schmidt games (cf. \cite{Wu2}). We prove the following result for a general $C^{1+\theta}$-expanding endomorphism:
\begin{theorem}\label{main3}
Let $f: M\to M$ be a $C^{1+\theta}$-expanding endomorphism. Then $E(f,y)$ is a winning set for modified Schmidt games induced by $f$ and played on $M$.
\end{theorem}

The paper is organized as follows. In Section 2, we will describe Schmidt games and two types of variants of them. Theorem \ref{main1}, \ref{main2} and \ref{main3} will be proved in Sections 3, 4 and 5 respectively. We always let $\nu$ denote the volume measure on $M$.

\section{Schmidt games}
\subsection{Schmidt games}
Let $(X, d)$ be a complete metric space. We denote as $B(x,r)$ the closed ball of radius $r$ with center $x$. If $\omega=(x,r) \in X \times \mathbb{R}_{+}$, we also denote $B(\omega):=B(x,r)$.

Schmidt games are played by two players, Alice and Bob. Fix $0 < \alpha, \beta <1$ and a subset $S \subset X$ (the target set). Bob starts the game by choosing $x_1 \in X$ and $r_1 >0$ hence specifying a pair $\omega_1=(x_1, r_1)$. Then Alice chooses a pair $\omega'_1=(x_1', r_1')$ such that $B(\omega'_1) \subset B(\omega_1)$ and $r_1'=\alpha r_1$. In the second turn, Bob chooses a pair $\omega_2=(x_2, r_2)$ such that $B(\omega_2) \subset B(\omega_1')$ and $r_2= \beta r_1'$, and so on. In the $k$th turn, Bob and Alice choose $\omega_k=(x_k, r_k)$ and $\omega'_k=(x_k', r_k')$ respectively such that
\begin{equation*}
  B(\omega'_k) \subset B(\omega_k) \subset B(\omega_{k-1}'), \ \ r_k= \beta r_{k-1}', \ \ r_k'=\alpha r_k.
\end{equation*}
Thus we have a nested sequence of balls in $X$:
\begin{equation*}
B(\omega_1) \supset B(\omega_1') \supset \cdots \supset B(\omega_k) \supset B(\omega_k') \supset \cdots.
\end{equation*}
The intersection of all these balls consists of a unique point $x_{\infty} \in X$. We call Alice the winner if $x_{\infty} \in S$, and Bob the winner otherwise. $S$ is called an \emph{$(\alpha, \beta)$-winning set} if Alice has a strategy to win regardless of how well Bob plays, and we call such a strategy an \emph{$(\alpha, \beta; S)$-winning strategy}. $S$ is called \emph{$\alpha$-winning} if it is $(\alpha, \beta)$-winning for any $0 <\beta <1$. $S$ is called a \emph{winning set} if it is $\alpha$-winning for some $0 < \alpha <1$.

The following nice properties of a winning set are proved in \cite{S}.

\begin{proposition}\label{winpro}(cf. \cite{S})
Some properties of winning sets for Schmidt games:
\begin{enumerate}
\item If the game is played on $X=\mathbb{R}^n$ with the Euclidean metric, then any winning set is dense and has full Hausdorff dimension $n$.
\item The intersection of countably many $\alpha$-winning sets is $\alpha$-winning.
\end{enumerate}
\end{proposition}

\subsection{Absolute games}
Fix $k\in \{0,1,\cdots, d-1\}$, $0 < \beta < 1/3$, and a subset $S \subset \mathbb{R}^d$ (the target set). We define the $k$-dimensional $\beta$-absolute games played on $\mathbb{R}^d$ as follows (cf. \cite{BFKRW}). Bob initially chooses $x_1 \in \mathbb{R}^d$ and $\rho_1>0$, hence specifying a closed ball $B_1=B(x_1, \rho_1)$. Then in each turn of play, after Bob chooses a closed ball $B_i=B(x_i, \rho_i)$, Alice chooses an affine subspace $L_i$ of dimension $k$ and removes its $\epsilon_i$-neighborhood $L_i^{(\epsilon_i)}$ from $B_i$ for some $0 < \epsilon_i \leq \beta \rho_i$. Then Bob chooses a closed ball $B_{i+1} \subset B_i\setminus L_i^{(\epsilon_i)}$ of radius $\rho_{i+1} \geq \beta\rho_i$. Alice wins the game if and only if
$$\bigcap_{i=1}^\infty B_i \cap S \neq \emptyset.$$
The set $S$ is said to be \emph{$k$-dimensionally $\beta$-absolute winning} if Alice has a winning strategy. We will say that $S$ is \emph{$k$-dimensionally absolute winning} if it is $k$-dimensionally $\beta$-absolute winning for every $0 < \beta< 1/3$.
\begin{proposition}\label{winningproperty}(cf. \cite{BFKRW})
Some properties of $k$-dimensional absolute winning set for any $0\leq k \leq d-1$:
\begin{enumerate}
\item $k$-dimensional absolute winning set implies $\alpha$-winning for all $0 < \alpha\leq 1/2$.
\item The countable intersection of $k$-dimensionally absolute winning sets is $k$-dimensionally absolute winning.
\item The image of a $k$-dimensionally absolute winning set under a $C^1$ diffeomorphism of $\mathbb{R}^d$ is
 $k$-dimensionally absolute winning.
\end{enumerate}
\end{proposition}
Observe that the strongest case when $k=0$, is precisely McMullen's absolute winning property (cf. \cite{Mc1}). So we call a $0$-dimensionally absolute winning set just an \emph{absolute winning set}. The notion of $k$-dimensionally absolute winning sets has been extended to subsets of $C^1$ manifolds in \cite {KW3}. It is more straightforward to extend the notion of $0$-dimensionally absolute winning sets to subsets of $C^1$ manifolds. In Theorem \ref{main2} we will be mostly interested in $0$-dimensionally absolute winning sets on manifolds.

Finally, we recall the potential games introduced in \cite{FSU} which are equivalent to absolute games in some cases, but for which it is easier to describe a winning strategy. Let $S\subset \mathbb{R}^d$ be a target set, and let $\beta \in (0, 1)$, $\gamma > 0$. The $k$-dimensionally $(\beta, \gamma)$-potential game is defined as follows: Bob begins the game by choosing a closed ball $B_1 \subset \mathbb{R}^d$. Then in each turn of play, after Bob chooses a closed ball $B_i$ of radius $\rho_i$, Alice chooses a countable family of neighborhoods of affine planes of dimension $k$, $\{L_{i,j}^{(\rho_{i,j})}: j \in \mathbb{N}\}$, such that
\begin{equation}\label{e:potentialcondition}
\sum_{j=1}^\infty\rho^\gamma_{i,j}\leq (\beta\rho_i)^\gamma.
\end{equation}
Then Bob chooses a closed ball $B_{i+1} \subset B_i$ of radius $\rho_{i+1} \geq \beta\rho_i$. Alice wins the game
if and only if
$$\bigcap_{i=1}^\infty B_i \cap\left(S\cup \bigcup_{i=1}^\infty\bigcup_{j=1}^\infty L_{i,j}^{(\rho_{i,j})}\right) \neq \emptyset.$$
The set $S$ is called \emph{$k$-dimensionally $(\beta, \gamma)$-potential winning} if Alice has a winning strategy, and is called \emph{$k$-dimensionally potential winning} if it is $k$-dimensionally $(\beta, \gamma)$-potential winning for any $\beta \in (0,1)$ and $\gamma > 0$. The following lemma follows from Theorem C.8 in \cite{FSU}.
\begin{lemma}\label{potential}
A subset $S \subset \mathbb{R}^d$ is $0$-dimensionally potential winning if and only if it is $0$-dimensionally absolute winning.
\end{lemma}
In Section 4, we will prove Theorem \ref{main2} by describing a winning strategy for $0$-dimensionally potential games.
\subsection{Modified Schmidt games induced by an expanding endomorphism}
In \cite{Wu2}, we have defined a type of modified Schmidt games induced by a partially hyperbolic diffeomorphism and played on any unstable manifold. In this section, we modify the construction in \cite{Wu2} to define a type of modified Schmidt games induced by an expanding endomorphism. We follow closely the notations used in \cite{Wu2}, and omit proofs here. Let $f: M \to M$ be a $C^{1+\theta}$-expanding endomorphism in this subsection.
\begin{definition}\label{tiling}
We always let $\epsilon>0$ be very small. If there exists a family of subsets $D_n^i$ ($n=1,2,\cdots $, $i=1,2,\cdots, k_n$) of $M$ such that
 \begin{enumerate}
\item each $D_n^i$ is an open and connected subset of $M$ satisfying
  $$B(f^n(x_n^i), \frac{\epsilon}{2}) \subset f^n(D_n^i) \subset B(f^n(x_n^i), \epsilon)$$
  for some $x_n^i \in D_n^i$,
\item $D_n^i \cap D_n^j=\emptyset$ for every $n$ and $i \neq j$,
\item $M= \bigcup_{i=1}^{k_n} \overline{D_n^i}$ for every $n$,
\end{enumerate}
then we say $\{D_n^i\}$ form a family of $f$-induced $\epsilon$-tilings on $M$. We denote it by $\mathcal{T}$ which consists of countable tilings $\{\mathcal{T}_n\}_{n=1}^\infty$ with $\mathcal{T}_n:=\{D_n^i\}_{i=1}^{k_n}$. Denote $\mathbf{T}_n=\bigcup_{i=1}^{k_n}D_n^i$ and $\mathbf{T}=\bigcup_{n=1}^\infty \mathbf{T}_n$.
\end{definition}

\begin{example}\label{exampletiling}
Let $S(\epsilon)$ be a maximal $\epsilon$-separated set in $M$. For each $z\in S(\epsilon)$, define
$$D_1(z):=\{y\in M: d(z,y)< d(w,y) \text{\ for any\ } w \in S(\epsilon) \text{\ and\ } w\neq z\}.$$
Let $\mathcal{T}_1:=\{D_1(z): z\in S(\epsilon)\}$. Note that $f^{-(n-1)}(D_1(z))$ consists of finitely many connected components for any $n\geq 1$. We collect all these connected components for every $D_1(z) \in \mathcal{T}_1$ and they form $\mathcal{T}_n$. Then one can verify that $\{\mathcal{T}_n\}_{n=1}^\infty$ form a family of $f$-induced $\epsilon$-tilings on $M$.
\end{example}

Now let us fix a family of $f$-induced $\epsilon$-tilings $\mathcal{T}$ on $M$. If $z\in \mathbf{T}_n$ we use $D_n^i(z)$ to denote the $n$-atom in $\mathcal{T}$ which contains $z$. We also write $\psi(z,n)=\overline{D_n^i(z)}$. We can define a partial order on $\Omega=\{(z,n):z\in \mathbf{T}_n, n\geq 1\}$:
$$(z,n)\leq (z',n') \Leftrightarrow \psi(z,n) \subset \psi(z',n').$$
The following properties of $\mathcal{T}$ is crucial in defining modified Schmidt games and establishing nice properties of winning sets.
\begin{flushleft}
\textbf{(MSG0)} There exists $a_* \in \mathbb{N}_+$ such that the following property holds: for any $(z,n)\in \Omega$ and any $m>a_*$ there exists $z'\in \mathbf{T}_{m+n}$ such that $(z',m+n) \leq (z,n)$.
\end{flushleft}
\begin{flushleft}
  \textbf{(MSG1)} For any nonempty open set $U \subset M$, there is $(z,n)\in \Omega$ such that $\psi(z,n) \subset U$.
\end{flushleft}
\begin{flushleft}
\textbf{(MSG2)} There exist $C, \sigma >0$, such that $\text{diam}(\psi(z,n)) \leq C\exp(-\sigma n)$ for all $(z,n)\in \Omega$.
\end{flushleft}
\begin{flushleft}
\textbf{($\nu 1$)} $\nu(\psi(z,n)) >0, \text{\ for any\ } (z,n)\in \Omega$.
\end{flushleft}
\begin{flushleft}
    \textbf{($\nu 2$)} For any $a > a_*$, there exists $c=c(a)> 0$ satisfying the following property: for any $\omega=(z,n)\in \Omega$ and any $b>a_*$, there exist $\theta_1=(z_1,n+b), \cdots, \theta_N=(z_N,n+b)\in \Omega$ such that
    \begin{enumerate}
  \item $\psi(\theta_1), \cdots, \psi(\theta_N) \subset \psi(\omega)$ and they are essentially disjoint;
  \item for every $\psi(\theta_i') \subset \psi(\theta_i)$ where $\theta_i'=(z_i',n+a+b)$ one has
  $$\nu(\bigcup_{i=1}^N\psi(\theta_i')) \geq c\nu(\psi(\omega)).$$
\end{enumerate}
\end{flushleft}

To prove \textbf{($\nu 2$)}, we use the following bounded distortion lemma. Thus we can not weaken the regularity of $f$ from $C^{1+\theta}$ to $C^1$.
\begin{lemma}\label{bd}(Bounded distortion lemma, cf. Theorem 3.2 in \cite{U})
Let $f: M \to M$ be a $C^{1+\theta}$-expanding endomorphism, and $c>0$ be small enough, $k \geq 0$. For any $z_1, z_2$ with $f^k(z_1)\in B(f^k(z_2),c)$, one has
$$\frac{1}{K} \leq \frac{\text{Jac}(f^k)(z_1)}{\text{Jac}(f^k)(z_2)} \leq K$$
for some $K=K(c)$, and $K \to 1$ as $c \to 0$.
\end{lemma}

Finally we define a type of modified Schmidt games induced by $f$ on $M$ with respect to $\mathcal{T}$ as follows. Fix $a, b \in \mathbb{N}_+$ both larger than $a_*$ and a subset $S$ of $M$(the target set). Bob starts the $(a,b)$-game by choosing a $\mathcal{T}_{n_1}$-atom $D_{n_1}(z_1)$ and hence specifying a pair $\omega_1=(z_1,n_1)\in \Omega$. By virtue of (MSG0), Alice can choose a pair $\omega'_1=(z'_1,n'_1)$ such that $\psi(\omega'_1)\subset \psi(\omega_1)$ and $n'_1=n_1+a$. In the second turn, Bob chooses a pair $\omega_2=(z_2, n_2)$ such that $\psi(\omega_2) \subset \psi(\omega_1')$ and $n_2= n_1'+b$, and so on. In the $k$th turn, Bob and Alice choose $\omega_k=(z_k, n_k)$ and $\omega'_k=(z_k', n_k')$ respectively such that
\begin{equation*}
  \psi(\omega'_k) \subset \psi(\omega_k) \subset \psi(\omega_{k-1}'), \ \ n_k= n_{k-1}'+b, \ \ n_k'=n_k+a.
\end{equation*}
Note that Bob and Alice can always make their choices by virtue of (MSG0). Thus we have a nested sequence of atoms in $\mathcal{T}$:
\begin{equation*}
\psi(\omega_1) \supset \psi(\omega_1') \supset \cdots \supset \psi(\omega_k) \supset \psi(\omega_k') \supset \cdots.
\end{equation*}
By (MSG2), the intersection of all these atoms consists of a unique point $z_{\infty} \in M$ as $M$ is a complete metric space. We call Alice the winner if $z_{\infty} \in S$, and Bob the winner otherwise. $S$ is called an \emph{$(a, b)$-winning set} for the \emph{modified Schmidt games induced by $f$} if Alice has a strategy to win regardless of how well Bob plays, and we call such a strategy an \emph{$(a, b; S)$-winning strategy}. $S$ is called \emph{$a$-winning} if it is $(a,b)$-winning for any $b >a_*$. $S$ is called a \emph{winning set} if it is $a$-winning for some $a > a_*$. Using properties (MSG 0-2) and ($\nu$ 1-2) of $\mathcal{T}$, we can prove (cf. \cite{Wu2}):
\begin{proposition}\label{winproperty}(cf. \cite{Wu2})
Some properties of winning sets for modified Schmidt games induced by $f$ on $M$ with respect to $\mathcal{T}$:
\begin{enumerate}
\item An $(a,b)$-winning set is dense in $M$.
\item The intersection of countably many $a$-winning sets is $a$-winning.
\item Let $S\subset M$ be an $a$-winning set. Then for any open $U\subset M$ one has $\dim_H(S \cap U) = \dim M$.
\end{enumerate}
\end{proposition}

\section{Proof of Theorem \ref{main1}}
Let $X$ be an $n$-dimensional smooth and compact Riemannian manifold, and $\nu$ be the volume measure on $X$. Then $\nu$ has the following properties (cf. \cite{Wu}):
\begin{enumerate}
  \item $\nu$ satisfies a power law, i.e. there exist positive numbers $c_1, c_2, \rho_0$ such that:
$$c_1\rho^n \leq \nu (B(z, \rho)) \leq c_2\rho^n \ \ \ \ \ \forall z\in X, \ \ \ \forall 0<\rho<\rho_0;$$
  \item $\nu$ is a Federer measure on $X$, i.e. there exist positive numbers $D, \rho_0$ such that
$$\nu (B(x, 2\rho)) < D \nu (B(x, \rho)), \ \ \  \forall x \in X, \ \ \forall 0 < \rho < \rho_0;$$
  \item there exist $\rho_0 >0$ and some $C>0$ such that
$$\nu(B(x_1, \rho)\cap B(x_2, \epsilon \rho)) \leq C\epsilon^n \nu(B(x_1, \rho))$$
for any $x_1, x_2 \in X$, $ \forall 0 < \rho < \rho_0, \ 0 < \epsilon <1$.
\end{enumerate}
The following lemma is crucial for Alice to make choices in Schmidt games.
\begin{lemma} \label{Ne}(cf. \cite{BFK}, \cite{Wu})
Let $C, D, \rho_0$ be as above, and
\begin{equation} \label{e:winning}
\begin{aligned}
0< \alpha < \frac{1}{2}\left(\frac{1}{CD}\right)^{\frac{1}{n}}.
\end{aligned}
\end{equation}
Then there exists $\epsilon = \epsilon(C, D) \in (0, 1)$, such that if $x_1 \in X$, $0 <\rho < \rho_0$, $y_1, y_2, \cdots, y_N$ are $N$ points in $X$, there exists $x_2 \in X$ such that
$$B(x_2, \alpha \rho) \subset B(x_1, \rho),$$
and
$$B(x_2, \alpha \rho)\cap B(y_i, \alpha \rho)= \emptyset$$
for at least $\lceil \epsilon N \rceil$ (the smallest integer $\geq \epsilon N$) of the points $y_i, \ 1\leq i \leq N$.
\end{lemma}

In the remaining of this section, we always let $f: M\to M$ be a $C^2$-Anosov diffeomorphism with conformality on unstable manifolds, as in Theorem \ref{main1}. Then there exist constants $\lambda <1<\sigma_1<\sigma_2$ and a nontrivial $Df$-invariant splitting of the tangent bundle $TM= E^s \oplus E^u$ into so-called stable and unstable distributions, such that all unit vectors $v^{\sigma} \in E_x^\sigma$ ($\sigma= s,u$) with $x\in M$ satisfy
\begin{equation*}
\|D_xfv^s\| < \lambda <1< \sigma_1<\|D_xfv^u\|<\sigma_2,
\end{equation*}
for some suitable Riemannian metric on $M$. It is well known that the distributions $E^s$ and $E^u$ are H\"{o}lder continuous over $M$. Moreover, $E^s$ and $E^u$ are integrable: there exist foliations $W^s$ and $W^u$ such that $TW^s=E^s$ and $TW^u=E^u$. We denote by $B^\sigma(x, r)$ the open ball inside $W^\sigma$ with respect to the Riemannian metric $d^\sigma$ induced from $M$ ($\sigma=s,u$). Fix $0 <\tau <1$. For any $z \in M$, we define
\begin{equation*}
B_l(z):=B(z, \tau^l),
\end{equation*}
\begin{equation*}
C_l(z):=\bigcup_{u \in B^s (z, \tau^l)}B^u(u, \tau^l),
\end{equation*}
and
\begin{equation*}
D_l(z):=\bigcup_{w \in B^u (z, \tau^l)}B^s(w, \tau^l).
\end{equation*}
\begin{lemma}\label{sets}
There exist $l_1,l_2 >0$ such that for any $z\in M$
\begin{enumerate}
  \item $C_{l}(z) \subset B_{l-l_1}(z)$ and $B_{l}(z) \subset C_{l-l_1}(z)$ for any $l \geq l_1$;
  \item $D_{l}(z) \subset B_{l-l_2}(z)$ and $B_{l}(z) \subset D_{l-l_2}(z)$ for any $l \geq l_2$.
\end{enumerate}
\end{lemma}
\begin{proof}
See the proof of Lemma 5.10 in \cite{Wu}. Then we use compactness of $M$ to obtain the universal $l_1, l_2>0$ independent of $z$.
\end{proof}

\begin{lemma}\label{absolutecontinuous}
Let $z_1, z_2\in M$ be two nearby points. Then there exists a constant $C>1$ such that for nearby $w_1, w_2 \in W^u(z_1)$, one has
$$\frac{1}{C}d^u(w_1,w_2)\leq d^u(h^s_{z_1z_2}(w_1),h^s_{z_1z_2}(w_2)) \leq Cd^u(w_1,w_2)$$
where $h^s_{z_1z_2}$ is the holonomy map of foliation $W^s$ from $W^u(z_1)$ to $W^u(z_2)$.
\end{lemma}
\begin{proof}
Since $f|_{W^u}$ is conformal, it follows from a similar argument as in the proof of Corollary 19.1.11 in \cite{KH} that the holonomy map of $W^s$ is $C^1$. Then the lemma follows.
\end{proof}

Let $B(z,k,c):=\{w\in M: d(f^i(w), f^i(z)) \leq c \text{\ \ for\ } \forall 0\leq i \leq k-1\}$ be the Bowen ball centered at $z$. We have the following bounded distortion lemma:
\begin{lemma}\label{BDconformal}(Bounded distortion lemma)
For any $z_2 \in B(z_1,k,c)$, one has
$$\frac{1}{K} \leq \frac{\|D_{z_1}f^k|_{W^u}\|}{\|D_{z_2}f^k|_{W^u}\|} \leq K$$
for some $K=K(c)$, and $K \to 1$ as $c \to 0$.
\end{lemma}
\begin{proof}
Since $z\mapsto D_zf|_{W^u}$ is H\"{o}lder continuous, $\log \|D_zf|_{W^u}\|$ is as well, and thus there exist $l>0$, $0 < \theta' <1$ such that
$$\left|\log \|D_{z_1}f|_{W^u}\|-\log \|D_{z_2}f|_{W^u}\|\right| \leq l (d(z_1, z_2))^{\theta'}$$
for nearby $z_1$ and $z_2$. By Proposition 6.4.16 in \cite{KH}, one has
\begin{equation*}
d(f^i(z_1), f^i(z_2)) \leq 2Cc\sigma^{\min(i, k-1-i)}, \ \ \text{for\ \ } \forall 0 \leq i \leq k-1.
\end{equation*}
for some $\sigma\geq \max(\lambda, \sigma_1^{-1})$ and $C>0$. Thus,
\begin{equation*}
\begin{aligned}
\left|\log \frac{\|D_{z_1}f^k|_{W^u}\|}{\|D_{z_2}f^k|_{W^u}\|}\right| &\leq \sum_{i=0}^{k-1}\left|\log \|D_{f^i(z_1)}f|_{W^u}\|-\log \|D_{f^i(z_2)}f|_{W^u}\|\right|
 \\ &\leq \sum_{i=0}^{k-1}l(d(f^i(z_1), f^i(z_2)))^{\theta'}
\\&\leq \sum_{i=0}^{k-1}l\left(2Cc\sigma^{\min(i, k-1-i)}\right)^{\theta'} \\
&\leq \frac{2l(2Cc)^{\theta'}}{1-\sigma^{\theta'}}.
\end{aligned}
\end{equation*}
Hence $$\frac{1}{K} \leq \frac{\|D_{z_1}f^k|_{W^u}\|}{\|D_{z_2}f^k|_{W^u}\|} \leq K,$$
where $K=\exp(\frac{2l(2Cc)^{\theta'}}{1-\sigma^{\theta'}})$.
\end{proof}
We call the set
\begin{equation*}
\Pi(c) := \Pi(y, c):= \bigcup_{z\in B^s(y,c/2)} B^u(z, c/2)
\end{equation*}
an \emph{open $c$-rectangle} centered at $y$ ($c$ is very small). Note that $\Pi(y, c)=C_l(y)$ when $\tau^l=c/2$. Let $I_k(z,c)$ be the connected component of $f^{-k}(\Pi(c)) \cap W^u(z)$ on $W^u(z)$ containing $z$, where $k \geq 0$.

\begin{proof}[Proof of Theorem \ref{main1}]
Let $\alpha_0$ satisfy \eqref{e:winning} with $n=\dim E^u$. Let $l_0 \in \mathbb{N}$ be such that $\tau^{l_0} \leq \alpha_0 <\tau^{l_0-1}$. Let $\alpha=\tau^{l_0+2l_2+1}$ and pick an arbitrary $0 < \beta<1$. We consider $(\alpha,\beta)$-Schmidt games.

Let $\epsilon$ be as in Lemma \ref{Ne}. Choose $r \in \mathbb{N}$ large enough such that
\begin{equation} \label{e:ner1}
(1-\epsilon)^rN <1, \text{\ where\ } N=\lfloor \frac{\log 2+r\log(\frac{1}{\alpha\beta})}{\log\sigma_1}\rfloor+1.
\end{equation}
Fix $L>0$ small. Regardless of the initial moves of Bob, Alice can make arbitrary moves waiting until Bob chooses a ball of radius sufficiently small. Hence without loss of generality, we may assume that $B(\omega_1)$ with $\omega_1=(x_1,\rho)$ has radius small enough satisfying $\rho\leq\frac{L}{100}$.
Choose $c'>0$ small enough such that:
\begin{enumerate}
  \item $1 < K=K(c') \leq 1+\eta<2$ where $\eta >0$ is very small;
  \item for any $z\in \Pi(c')$, $B^u(z,L) \cap \Pi(c')$ has only one connected component which contains $z$;
  \item for any $z\in \Pi(c')$, $B^s(z,L) \cap \Pi(c')$ has only one connected component which contains $z$.
\end{enumerate}
Now choose $0 <c\ll c'$ such that:
\begin{equation}\label{e:c1}
c\leq \frac{\alpha c'(\alpha\beta)^{2r-1}}{100}
\end{equation}
and
\begin{equation}\label{e:j=01}
c<\frac{\alpha \rho(\alpha\beta)^{2r-1}}{C}
\end{equation}
where $C$ is from Lemma \ref{absolutecontinuous} when applied to any $z_1,z_2 \in B(\omega_1)$.
Note that the choice of $c$ depends heavily on $\rho$, i.e. the initial move of Bob. Suppose that $\tau^{m} \leq c/2< \tau^{m-1}$. Then $D_{m+l_1+l_2}(y) \subset B_{m+l_1}(y) \subset C_{m}(y) \subset \Pi(c)$ by Lemma \ref{sets}.

Now we describe a strategy for Alice to win $(\alpha, \beta)$-Schmidt games played on $M$ with target set $S=E(f, y)$. We claim that for each $j\in \mathbb{N}$, Alice can ensure that $x \notin f^{-k}(D_{m+l_1+l_2}(y))$ for any $x\in B(\omega'_{r(j+1)})$ and any $k$ with
\begin{equation}\label{e:somek}
\sup\{\text{diam\ }(I_k(z,c))\} > \frac{\alpha \rho(\alpha\beta)^{(j+2)r-1}}{C}
\end{equation}
where the supremum is taken over all $I_k(z,c)$'s with $z\in f^{-k}(D_{m+l_1+l_2}(y)) \cap D_{l+l_2}(x_{jr+1})$. Here $l=l(j)$ is such that $\tau^{l}\leq \rho(\alpha\beta)^{jr}<\tau^{l-1}$. Moreover, Alice can also ensure that $B(\omega'_{jr+1}) \subset D_{l+l_2}(x_{jr+1})$, which implies that $\bigcap_iB(\omega'_i) = \bigcap_jD_{l(j)+l_2}(x_{jr+1})$. This will imply that
$$\bigcap_iB(\omega'_i) \subset (\bigcup_{k}f^{-k}(D_{m+l_1+l_2}(y)))^c \subset E(f, y)$$
and finish the proof of the theorem. Indeed, if $z\in \bigcap_iB(\omega'_i) = \bigcap_jD_{l(j)+l_2}(x_{jr+1})$ and meanwhile $z\in f^{-k}(D_{m+l_1+l_2}(y))$ for some $k\geq 0$, there exists some $j$ such that \eqref{e:somek} is satisfied. However, Alice has ensured that $f^{-k}(D_{m+l_1+l_2}(y))\cap B(\omega'_{r(j+1)})=\emptyset$, which gives a contradiction.

We prove the claim by induction on $j$. At $j=0$ step, by \eqref{e:j=01} one has for any $k \in\mathbb{N}$,
$$\text{diam\ }(I_k(z,c)) \leq c < \frac{\alpha \rho(\alpha\beta)^{2r-1}}{C}.$$
So there is no $f^{-k}(D_{m+l_1+l_2}(y))$ for Alice to avoid. She just needs to ensure that $B(\omega'_{1}) \subset D_{l+l_2}(x_{1})$ where $l$ is such that $\tau^{l}\leq \rho<\tau^{l-1}$. But this is guaranteed by the choice of $\alpha$.

Assume the claim is true for $0, 1, \cdots, j-1$. Now we consider the $j$th step. Suppose that Bob already picked $B(\omega_{jr+1})$. In this step (containing $r$ turns of play), Alice only needs to avoid the $f^{-k}(D_{m+l_1+l_2}(y))$'s with $k$ satisfying
\begin{equation} \label{e: intersection1}
f^{-k}(D_{m+l_1+l_2}(y)) \cap D_{l+l_2}(x_{jr+1})  \neq \emptyset,
\end{equation}
where $l$ is such that $\tau^{l}\leq \rho(\alpha\beta)^{jr}<\tau^{l-1}$, and
\begin{equation} \label{e: mainequality1}
\frac{\alpha \rho(\alpha\beta)^{(j+2)r-1}}{C} < \sup\{\text{diam\ }(I_k(z,c))\} \leq \frac{\alpha \rho(\alpha\beta)^{(j+1)r-1}}{C}
\end{equation}
where the supremum is taken over all $I_k(z,c)$'s with $z\in f^{-k}(D_{m+l_1+l_2}(y)) \cap D_{l+l_2}(x_{jr+1})$.

The following lemma will allow us to apply bounded distortion Lemma \ref{BDconformal} in Lemma \ref{manyk1} below.
\begin{lemma} \label{bowenball}
Let $k_1 \leq k_2$ satisfy both \eqref{e: intersection1} and \eqref{e: mainequality1}. Suppose that
$$z_i \in f^{-k_i}(D_{m+l_1+l_2}(y)) \cap D_{l+l_2}(x_{jr+1}).$$
Moreover, assume that
$$\text{diam\ }(I_{k_1}(z_1,c))>\frac{\alpha \rho(\alpha\beta)^{(j+2)r-1}}{C}.$$
Then $z_2 \in B(z_1, k_1, c').$
\end{lemma}
\begin{proof}[Proof of Lemma \ref{bowenball}]
Recall the choice of $c$ in \eqref{e:c1} and the definition of $\Pi(c)$. It is clear that
\begin{equation*}
\text{diam\ }(f^{k_1}(I_{k_1}(z_1,c)))=c, \ \ \text{diam\ }(f^{k_1}(I_{k_1}(z_1,c')))=c'.
\end{equation*}
By \eqref{e:c1} and bounded distortion Lemma \ref{BDconformal} we have
\begin{equation*}
\frac{\text{diam\ }(I_{k_1}(z_1,c'))}{\text{diam\ }(I_{k_1}(z_1,c))} \geq \frac{c'}{cK}\geq \frac{100}{K\alpha(\alpha\beta)^{2r-1}}.
\end{equation*}
By assumption,
\begin{equation*}
\text{diam\ }(I_{k_1}(z_1,c)) > \frac{\alpha \rho(\alpha\beta)^{(j+2)r-1}}{C}.
\end{equation*}
If $C$ is sufficiently close to $1$, then
\begin{equation*}
\text{diam\ }(I_{k_1}(z_1,c'))> \frac{100}{K\alpha(\alpha\beta)^{2r-1}}\cdot \frac{\alpha \rho(\alpha\beta)^{(j+2)r-1}}{C}=\frac{100\rho(\alpha\beta)^{jr}}{CK}>50\rho(\alpha\beta)^{jr}.
\end{equation*}
This means that $f^{k_1}(B^u(z_1,10\rho(\alpha\beta)^{jr})) \subset \Pi(c')$. As $f$ expands $E^u$, this implies $\text{diam}(f^{k}(B^u(z_1,10\rho(\alpha\beta)^{jr})))\leq c'$ for any $0\leq k \leq k_1$.

We define the map $h^s_{z}(w):=h^s_{wz}(w)$ to be the projection of $w$ onto $W^u(z)$ along foliation $W^s$. It is clear that $h^s_{z_1}(z_2)\in B^u(z_1,2\rho(\alpha\beta)^{jr})$. Note that for any $0\leq k \leq k_1$
\begin{equation*}
\begin{aligned}
d(f^{k}(z_2),f^{k}(h^s_{z_1}(z_2)))&\leq d^s(f^{k}(z_2),f^{k}(h^s_{z_1}(z_2)))\\
&\leq \lambda^{k}d^s(z_2,h^s_{z_1}(z_2))\\
& \leq \lambda^{k}(2\rho(\alpha\beta)^{jr})\\
&\ll \text{diam}\left(f^{k}(B^u(z_1,2\rho(\alpha\beta)^{jr}))\right).
\end{aligned}
\end{equation*}
Thus
\begin{equation*}
\begin{aligned}
d(f^{k}(z_1),f^{k}(z_2)) &\leq d(f^{k}(z_1),f^{k}(h^s_{z_1}(z_2)))+d(f^{k}(h^s_{z_1}(z_2),f^{k}(z_2)))\\
&\leq  2\text{diam}\left(f^{k}(B^u(z_1,2\rho(\alpha\beta)^{jr}))\right)\leq c'.
\end{aligned}
\end{equation*}
Thus $z_2\in B(z_1,k_1,c')$. This proves the lemma.
\end{proof}

\begin{lemma}\label{onecomponent1}
For each $k$ satisfying \eqref{e: intersection1} and \eqref{e: mainequality1}, $f^{-k}(D_{m+l_1+l_2}(y)) \cap D_{l+l_2}(x_{jr+1})$ has at most one connected component.
\end{lemma}
\begin{proof}[Proof of Lemma \ref{onecomponent1}]
Assume the contrary. Let $z_1, z_2$ be points in two different connected components of $f^{-k}(D_{m+l_1+l_2}(y)) \cap D_{l+l_2}(x_{jr+1})$. Moreover, we let $z_1$ be such that
 $$\text{diam\ }(I_{k}(z_1,c))>\frac{\alpha \rho(\alpha\beta)^{(j+2)r-1}}{C}.$$
Then we can apply the proof of Lemma \ref{bowenball} here with $k_1=k_2=k$. There exists a unique point $w\in B^u(z_1)\cap B^s(z_2)$, i.e., $w=h^s_{z_1}(z_2)$ as defined in the proof of Lemma \ref{bowenball}. By the proof of Lemma \ref{bowenball}, $d^u(f^k(w),f^k(z_1)) \leq c'$ and  $d^s(f^k(w),f^k(z_2)) \leq c'$. But $f^k(z_2) \in \Pi(c)$, which implies that $d^s(f^k(w),f^k(z_2)) \geq L$ by the choice of $c'$ and $L$. We arrive at a contradiction.
\end{proof}
By Lemma \ref{onecomponent1}, the supremum in \eqref{e: mainequality1} is in fact taken over the unique connected component of $f^{-k}(D_{m+l_1+l_2}(y)) \cap D_{l+l_2}(x_{jr+1})$.
\begin{lemma} \label{manyk1}
There are at most $N$ many $k$'s satisfying both \eqref{e: intersection1} and \eqref{e: mainequality1}.
\end{lemma}
\begin{proof}[Proof of Lemma \ref{manyk1}]
Let $k_1$ and $k_2$ be the minimal and the maximal ones respectively among all $k$'s satisfying both \eqref{e: intersection1} and \eqref{e: mainequality1}.

For any $w\in f^{-k_2}(D_{m+l_1+l_2}(y)) \cap D_{l+l_2}(x_{jr+1})$ one has
\begin{equation} \label{e: k1/k21}
\text{diam\ }(f^{k_1}(I_{k_2}(w,c))) \leq \frac{\text{diam\ }(f^{k_2}(I_{k_2}(w,c)))}{(\sigma_1)^{k_2-k_1}}= \frac{c}{(\sigma_1)^{k_2-k_1}}.
\end{equation}
Lemma \ref{bowenball} allows us to apply bounded distortion Lemma \ref{BDconformal} to get
\begin{equation} \label{e: k2//k11}
\frac{\text{diam\ }(I_{k_2}(w,c))}{\text{diam\ }(I_{k_1}(z,c))} \leq \frac{K\text{diam\ }(f^{k_1}(I_{k_2}(w,c)))}{\text{diam\ }(f^{k_1}(I_{k_1}(z,c)))} \leq \frac{K}{(\sigma_1)^{k_2-k_1}}
\end{equation}
where the latter inequality follows from \eqref{e: k1/k21} and $\text{diam\ }\left(f^{k_1}(I_{k_1}(z,c))\right)=c$. Combining \eqref{e: mainequality1}, \eqref{e: k2//k11} and bounded distortion Lemma \ref{BDconformal}, one has
\begin{equation*}
\begin{aligned}
\frac{\alpha \rho(\alpha\beta)^{(j+2)r-1}}{C} &< \text{diam\ }(I_{k_2}(w,c))  \\
&\leq \frac{K}{(\sigma_1)^{k_2-k_1}}\text{diam\ }(I_{k_1}(z,c)) \\
&\leq  \frac{K}{(\sigma_1)^{k_2-k_1}}\cdot\frac{\alpha \rho(\alpha\beta)^{(j+1)r-1}}{C},
\end{aligned}
\end{equation*}
which implies
\begin{equation*}
(\sigma_1)^{k_2-k_1} \leq \frac{K}{(\alpha\beta)^r}.
\end{equation*}
Hence
\begin{equation*}
k_2-k_1 \leq \lfloor \frac{\log K+r\log(\frac{1}{\alpha\beta})}{\log\sigma_1}\rfloor \leq N-1
\end{equation*}
which finishes the proof of the lemma.
\end{proof}
Now we project all $f^{-k}(D_{m+l_1+l_2}(y)) \cap D_{l+l_2}(x_{jr+1})$ satisfying both \eqref{e: intersection1} and \eqref{e: mainequality1} onto $B^u(x_{jr+1}, \tau^{l+l_2})$ along the foliation $W^s$. By Lemma \ref{absolutecontinuous} and \eqref{e: mainequality1}, the projection of each of them is contained in a ball of diameter at most
$$\alpha \rho(\alpha\beta)^{(j+1)r-1} < \alpha\tau^{\tilde{l}-1}< \tau^{l_0+\tilde{l}+l_2}\leq \alpha_0\tau^{\tilde{l}+l_2}$$
where $\tilde{l}$ is such that $\tau^{\tilde{l}} \leq \rho(\alpha\beta)^{(j+1)r-1}<\tau^{\tilde{l}-1}$.
Note that there are at most $N$ many such sets. Alice can apply Lemma \ref{Ne} to choose $B^u(x'_{jr+1},\tau^{l+l_2}\alpha_0)\subset B^u(x_{jr+1}, \tau^{l+l_2})$ to avoid at least $\lceil \epsilon N \rceil$ of the above sets. It is clear that the set
$$\bigcup_{z\in B^u(x'_{jr+1},\tau^{l+l_2}\alpha_0)}B^s(z,\tau^{l+l_2}\alpha_0)$$
has empty intersection with at least $\lceil \epsilon N \rceil$ connected components. Let Alice choose the ball $$B(x'_{jr+1},\rho(\alpha\beta)^{jr}\alpha) \subset D_{l+l_2+l_0}(x'_{jr+1}) \subset \bigcup_{z\in B^u(x'_{jr+1},\tau^{l+l_2}\alpha_0)}B^s(z,\tau^{l+l_2}\alpha_0).$$
Then $B(x'_{jr+1},\rho(\alpha\beta)^{jr}\alpha)$ has empty intersection with at least $\lceil \epsilon N \rceil$ connected components. Moreover, $B(x'_{jr+1},\rho(\alpha\beta)^{jr}\alpha) \subset D_{l+l_2}(x_{jr+1})$. Repeat the above argument $r$ times. Since $N(1-\epsilon)^r <1$ by \eqref{e:ner1}, Alice can avoid all the $f^{-k}(D_{m+l_1+l_2}(y))$'s satisfying both \eqref{e: intersection1} and \eqref{e: mainequality1} after $r$ turns of play. This proves the claim. Continue with the induction and the theorem follows.
\end{proof}

Due to the stability of winning property under countable intersections, we have:
\begin{corollary}\label{corollary}
Let $\{f_i\}_{i=1}^N$ be a finite set of $C^2$-Anosov diffeomorphisms with conformality on unstable manifolds, as in Theorem \ref{main1}. And let $Y$ be a countable subset of $M$. Then $\cap_{i=1}^N E(f_i,Y)$ is winning for Schmidt games played on $M$.
\end{corollary}
\begin{proof}
For each diffeomorphism $f_i$, $E(f_i,y)$ is $\alpha_i$-winning for Schmidt games where $\alpha_i$ depends on $l_2$ in Lemma \ref{sets} according to the proof of Theorem \ref{main1}. Thus $\cap_{i=1}^N E(f_i,y)$ is $\alpha$-winning where $\alpha:=\min\{\alpha_1, \cdots, \alpha_N\}$, for any $y \in M$. By (2) of Proposition \ref{winpro}, $\cap_{i=1}^N E(f_i,Y)$ is $\alpha$-winning for Schmidt games.
\end{proof}

In particular, Corollary \ref{corollary} implies that if $f$ is a $C^2$-Anosov diffeomorphism with conformality both on unstable and stable manifolds, then $E(f,y)\cap E(f^{-1},y)$ is winning for Schmidt games played on $M$, and hence so is the set of points with non-dense \emph{complete} orbit.

Since the parameter $\alpha$ of the winning set $E(f,Y)$ for Schmidt games depends on the diffeomorphism $f$ itself in Theorem \ref{main1}, we don't know whether Corollary \ref{corollary} holds true for a countable set of $C^2$-Anosov diffeomorphisms with conformality on unstable manifolds.

\section{Proof of Theorem \ref{main2}}
In this section we prove Theorem \ref{main2}. Let $f: M\to M$ be a $C^{1+\theta}$-expanding endomorphism. Then $f$ is a local homeomorphism, i.e., there exists $\rho_0>0$ such that for any $z\in M$, $f|_{B(z,\rho_0)}$ is a homeomorphism onto its image. In this section, we suppose that $f$ is conformal. We also suppose that $1< \sigma_1\leq \|D_zf\| \leq \sigma_2$ for any $z\in M$. The following version of bounded distortion lemma is also standard:
\begin{lemma}\label{BDconformal2}(Bounded distortion lemma, cf. Theorem 3.2 in \cite{U})
Let $f: M\to M$ be a $C^{1+\theta}$-expanding endomorphism and suppose that $f$ is conformal. For any $z_2 \in B(z_1,k,c)$, one has
$$\frac{1}{K} \leq \frac{\|D_{z_1}f^k\|}{\|D_{z_2}f^k\|} \leq K$$
for some $K=K(c)$, and $K \to 1$ as $c \to 0$.
\end{lemma}

\begin{proof}[Proof of Theorem \ref{main2}]
By Lemma \ref{potential}, it suffices to prove that $E(f,y)$ is $0$-dimensionally $(\beta, \gamma)$-potential winning for any $\beta \in (0,1)$ and $\gamma > 0$. Choose $r \in \mathbb{N}$ large enough such that
\begin{equation} \label{e:ner}
N\beta^{(r-1)\gamma}\leq 1
\end{equation}
where $N=\lfloor \frac{\log 2+r\log(\frac{1}{\beta})}{\log\sigma_1}\rfloor+1$. Fix $c'>0$ to be very small such that
\begin{enumerate}
  \item $1 < K=K(c') \leq 1+\eta<2$ where $\eta >0$ is very small,
  \item $f|_{B(z,c')}$ is injective for any $z\in M$.
\end{enumerate}
Regardless of the initial moves of Bob, Alice can make empty moves waiting until Bob chooses a ball of radius sufficiently small. Hence without loss of generality, we may assume that $B(\omega_1)$ has radius $\rho_1\leq \frac{c'}{100}$. Without loss of generality, we assume that Bob will play so that $\rho_i \to 0$ (cf. Remark 3.2 in \cite{KW3}). Now choose $0 <c\ll c'$ such that:
\begin{equation}\label{e:c}
c\leq \frac{c'\beta^{2r}}{100K}
\end{equation}
and
\begin{equation}\label{e:j=0}
c<\rho_1\beta^{2r}.
\end{equation}
Note that the choice of $c$ depends heavily on $\rho_1$, i.e. the initial move of Bob. If $f^k(z)\in B(y,c)$, we let $I_k(z,c)$ be the connected component of $f^{-k}(B(y,c))$ which contains $z$. Let $I_k(c)$ mean any one of the connected components of $f^{-k}(B(y,c))$.

Now we describe a strategy for Alice to win the $0$-dimensionally $(\beta, \gamma)$-potential games on $M$ with target set $S=E(f, y)$. Suppose that each step contains $r$ turns of play. Let $j\in \mathbb{N}$. Suppose that Bob has chosen a ball $B_{jr+1}$ of radius $\rho_{jr+1}$ at $(jr+1)$th turn of play, i.e. at the beginning of $j$th step. Then we can let Alice choose all the balls of radius $\text{diam\ }(I_k(c))$ containing $I_k(c)$ which satisfies:
\begin{equation} \label{e: mainequality}
\rho_{jr+1}\beta^{2r}\leq \text{diam\ }(I_k(c)) < \rho_{jr+1}\beta^{r}
\end{equation}
and
\begin{equation} \label{e: intersection}
I_k(c) \cap B_{jr+1} \neq \emptyset
\end{equation}
at $(jr+1)$th turn of play and let Alice just make empty moves at the remaining $(r-1)$ turns of play at $j$th step. Note that by \eqref{e:j=0} one has for any $k \in\mathbb{N}$,
$$\text{diam\ }(I_k(c)) \leq c < \rho_1\beta^{2r}.$$
So Alice just makes empty moves at $j=0$ step. For other $j\in \mathbb{N}$, let us check condition \eqref{e:potentialcondition} to guarantee that Alice's move is legal at $(jr+1)$th turn of play.

\begin{lemma}\label{onek}
For each $k$, there exists at most one $I_k(c)$ satisfying both \eqref{e: mainequality} and \eqref{e: intersection}.
\end{lemma}

\begin{proof}[Proof of Lemma \ref{onek}]
Assume that $I_k(c), I'_k(c)$ are two different connected components satisfying both \eqref{e: mainequality} and \eqref{e: intersection}. Then both $I_k(c)$ and $I'_k(c)$ are contained in the ball $B_{jr+1}'$ of radius $\rho_{jr+1}(1+\beta^{r})$ and concentric with $B_{jr+1}$. By bounded distortion Lemma \ref{BDconformal2} and \eqref{e:c},
$$\frac{\text{diam\ }(I_k(c'))}{\text{diam\ }(I_k(c))} \geq \frac{c'}{Kc}\geq \frac{100}{\beta^{2r}},$$
where $I_k(c')$ is the connected components of $f^{-k}(B(y,c')$ containing $I_k(c)$. Combining the above and \eqref{e: mainequality}, we have
$$\text{diam\ }(I_k(c'))\geq \frac{100}{\beta^{2r}}\cdot \rho_{jr+1}\beta^{2r}=100\rho_{jr+1}.$$
This implies that $B_{jr+1}' \subset f^{-k}(B(y,c'))$. In other words, $f^k(B'_{jr+1}) \subset B(y,c')$. By the choice of $c'$, we know that it is impossible that $f^k(I_k(c))=f^k(I'_k(c))$. We arrive at a contradiction.
\end{proof}

\begin{lemma} \label{manyk}
There are at most $N$ many $k$'s satisfying both \eqref{e: mainequality} and \eqref{e: intersection}.
\end{lemma}
\begin{proof}[Proof of Lemma \ref{manyk}]
Let $k_1$ and $k_2$ be the minimal and the maximal ones respectively among all $k$'s satisfying both \eqref{e: mainequality} and \eqref{e: intersection}. Then $k_1 \leq k_2$. It is clear that
\begin{equation} \label{e: k1k2}
\text{diam\ }(f^{k_1}(I_{k_1}(c))) = \text{diam\ }(f^{k_2}(I_{k_2}(c)))=c,
\end{equation}
and
\begin{equation} \label{e: k1/k2}
\text{diam\ }(f^{k_1}(I_{k_2}(c))) \leq \frac{\text{diam\ }(f^{k_2}(I_{k_2}(c)))}{(\sigma_1)^{k_2-k_1}}= \frac{c}{(\sigma_1)^{k_2-k_1}}.
\end{equation}
Moreover, by the proof of Lemma \ref{onek} we have $B_{jr+1}' \subset f^{-k_1}(B(y,c'))$. As $I_{k_2}(c) \subset B_{jr+1}'$, one has $f^{k_1}(I_{k_2}(c)) \subset B(y,c')$. Thus we can apply bounded distortion Lemma \ref{BDconformal2} to get
\begin{equation} \label{e: k2//k1}
\frac{\text{diam\ }(I_{k_2}(c))}{\text{diam\ }(I_{k_1}(c))} \leq \frac{K\text{diam\ }(f^{k_1}(I_{k_2}(c)))}{\text{diam\ }(f^{k_1}(I_{k_1}(c)))} \leq \frac{K}{(\sigma_1)^{k_2-k_1}}
\end{equation}
where the latter inequality follows from \eqref{e: k1k2} and \eqref{e: k1/k2}. Combining \eqref{e: mainequality} and \eqref{e: k2//k1}, one has

\begin{equation*}
\begin{aligned}
\rho_{jr+1}\beta^{2r} &\leq \text{diam\ }(I_{k_2}(c))  \\
&\leq \frac{K}{(\sigma_1)^{k_2-k_1}}\text{diam\ }(I_{k_1}(c)) \\
&\leq  \frac{K}{(\sigma_1)^{k_2-k_1}}\rho_{jr+1}\beta^{r},
\end{aligned}
\end{equation*}
which implies:
\begin{equation*}
(\sigma_1)^{k_2-k_1} \leq \frac{K}{\beta^r}.
\end{equation*}
Hence
\begin{equation*}
k_2-k_1 \leq \lfloor \frac{\log K+r\log(\frac{1}{\beta})}{\log\sigma_1}\rfloor \leq N-1
\end{equation*}
which finishes the proof of the lemma.
\end{proof}
By \eqref{e:ner} and \eqref{e: mainequality}, we have
$$\sum_{k=1}^\infty\rho^\gamma_{jr+1,k}\leq N(\rho_{jr+1}\beta^{r})^\gamma \leq (\beta\rho_{jr+1})^\gamma$$
which verifies \eqref{e:potentialcondition}. Now let $z\in \cap_{i=1}^\infty B_i$ and suppose that $f^k(z) \in B(y,c)$ for some $k\geq 0$. Then $z$ is contained in some $I_k(c)$. Since $\rho_{i+1} \geq \beta\rho_i$, we have that $\rho_{jr+1}\beta^{2r} \leq \rho_{(j+1)r+1}\beta^r$. This means that the union of Alice's choices, $\bigcup_{i=1}^\infty\bigcup_{k=1}^\infty L_{i,k}^{(\rho_{i,k})}$ must contain all possible $I_k(c)$'s where $z$ lies in. Thus $z\in \bigcup_{i=1}^\infty\bigcup_{k=1}^\infty L_{i,k}^{(\rho_{i,k})}$. This implies
$$\bigcap_{i=1}^\infty B_i \cap\left(E(f,y)\cup \bigcup_{i=1}^\infty\bigcup_{k=1}^\infty L_{i,k}^{(\rho_{i,k})}\right) \neq \emptyset.$$
So Alice wins the game.
\end{proof}

\begin{proof}[Proof of Corollary \ref{coro2}]
It follows from Theorem \ref{main2} and (2) of Proposition \ref{winningproperty}.
\end{proof}
\section{Proof of Theorem \ref{main3}}
In this section, we let $f: M\to M$ be a $C^{1+\theta}$-expanding endomorphism and prove Theorem \ref{main3}. Suppose that $1< \sigma_1\leq \|D_zf\| \leq \sigma_2$ for any $z\in M$. Alice will make choices in modified Schmidt games defined in Subsection 2.3 according to the following lemma:
\begin{lemma}\label{avoid}(cf. Lemma 3.1 in \cite{Wu2})
Let $a\in \mathbb{N}_+$ be such that
\begin{equation}\label{e:a}
a > \frac{\log 13}{\log\sigma_1}.
\end{equation}
Then there exists $0<\eta<\frac{1}{4}$ such that if $(z,n) \in \Omega$, and $B_1, \cdots, B_N$ are $N$ subsets in $D_0$ with $\text{diam}(B_i) < \frac{\epsilon}{2\sigma_2^{n+a}}$, there exists $(z', n+a)\in \Omega$ such that
$$\psi(z',n+a) \subset \psi(z,n),$$
and
$$\psi(z',n+a)\cap B_i=\emptyset$$
for at least $\lceil \eta N \rceil$ (the smallest integer $\geq \eta N$) of the sets $B_i, \ 1\leq i \leq N$. \end{lemma}

\begin{proof}[Proof of Theorem \ref{main3}]
Let $a>a_*$ satisfy \eqref{e:a} and let $b>a_*$ be arbitrary. Let $r\in \mathbb{N}$ be large enough with
$$(1-\eta)^r(a+b)r<1.$$
Fix $L>0$ to be very small such that $f|_{B(z,L)}$ is injective for any $z\in M$. Regardless of the initial moves of Bob, Alice can make arbitrary moves waiting until Bob chooses an atom $\psi(z_1, n_1)$ with $n_1$ large enough such that
\begin{equation}\label{e:n1}
\frac{2\epsilon}{\sigma_1^{n_1-(a+b)r}}\leq \frac{L}{100}.
\end{equation}
Let $c>0$ be small enough such that $c \leq \frac{\epsilon}{2\sigma_2^{n_1+(a+b)r+a}}$. Note that $f^{-k}(B(y,c))$ consists of finitely many connected components. If $f^k(z)\in B(y,c)$ we let $I_k(z,c)$ be the connected component which contains $z$. Let $I_k(c)$ mean any one of the connected components of $f^{-k}(B(y,c))$.

Now we describe a strategy for Alice to win the $(a, b)$-modified Schmidt games induced by $f$ and played on $M$ with target set $S=E(f, y)$. We claim that for each $j\in \mathbb{N}$, Alice can ensure that for any $x\in \psi(\omega'_{r(j+1)})$ and any $I_k(c)$ with $k< (j+1)r(a+b)$, she has $x \notin I_k(c)$. This will imply $\cap_i\psi(\omega'_i) \subset (\bigcup_{k} I_k(c))^c \subset E(f, y)$ and finish the proof.

Let us prove the claim by induction on $j$. Note that each step consists of $r$ turns of play. Consider $j=0$. Suppose that Bob has chosen an atom $\psi(z_1,n_1)$. So Alice needs to avoid all $I_k(c)$'s with $0\leq k< (a+b)r$ in the next $r$ turns of play. For a given $0\leq k< (a+b)r$, we have
$$\text{diam}(f^k(\psi(z_1,n_1)))\leq \frac{2\epsilon}{\sigma_1^{n_1-k}} \leq\frac{2\epsilon}{\sigma_1^{n_1-(a+b)r}}\leq \frac{L}{100}$$
by \eqref{e:n1}. This implies that $\psi(z_1,n_1) \cap f^{-k}(B(y,c))$ has at most one connected component by the choice of $c$ and $L$. In other words, for each $0\leq k< (a+b)r$ there is at most one $I_k(c)$ intersecting with $\psi(z_1,n_1)$. So there are at most $(a+b)r$ many $I_k(c)$'s intersecting with $\psi(z_1,n_1)$. For each of them one has
$$\text{diam}(I_k(c)) \leq \frac{c}{\sigma_1^k} \leq c \leq \frac{\epsilon}{2\sigma_2^{n_1+(a+b)r+a}}$$
by the choice of $c$. This guarantees that Alice can apply Lemma \ref{avoid} in each of the $r$ turns of play to avoid some of these $I_k(c)$'s. Since $r(a+b)(1-\eta)^r<1$, after $r$ turns of play Alice can avoid all these $I_k(c)$'s. So the claim is true when $j=0$.

Assume the claim is true for $0, 1, \cdots, j-1$. Now we consider the $j$th step. Suppose that Bob already picked $\psi(\omega_{jr+1})$. In this step Alice only needs to avoid the $I_k(c)$'s satisfying
\begin{equation}\label{e:ks}
jr(a+b) \leq k < (j+1)r(a+b)
\end{equation}
and
\begin{equation} \label{e:intersect}
I_k(c) \cap \psi(\omega_{jr+1})  \neq \emptyset.
\end{equation}
For a given $k$ with \eqref{e:ks},
\begin{equation*}
\begin{aligned}
\text{diam}(f^k(\psi(\omega_{jr+1}))) &= \text{diam}(f^k(\overline{D_{n_1+jr(a+b)}^i}))\\
&\leq \frac{2\epsilon}{\sigma_1^{n_1+jr(a+b)-k}} \\
&\leq\frac{2\epsilon}{\sigma_1^{n_1-(a+b)r}}\\
&\leq \frac{L}{100}
\end{aligned}
\end{equation*}
by \eqref{e:n1} and \eqref{e:ks}. This implies that $\psi(\omega_{jr+1}) \cap f^{-k}(B(y,c))$ has at most one connected component by the choice of $c$ and $L$. So there are at most $r(a+b)$ many $I_k(c)$'s satisfying \eqref{e:ks} and \eqref{e:intersect}.

We have at most  $r(a+b)$ many $f^{jr(a+b)}(I_k(c))$'s intersecting $f^{jr(a+b)}(\psi(\omega_{jr+1}))$ by applying $f^{jr(a+b)}$. As $0\leq k-jr(a+b)< r(a+b)$ and $f^{jr(a+b)}(\psi(\omega_{jr+1}))$ is an $n_1$-atom, we can use the same argument as in the case $j=0$ for Alice to get a choice in the picture of $f^{jr(a+b)}(\psi(\omega_{jr+1}))$ first, and then applying $f^{-jr(a+b)}$ back to get a choice in the picture of $\psi(\omega_{jr+1})$ in each turn of play. After $r$ turns of play Alice can avoid all these $I_k(c)$'s satisfying \eqref{e:ks} and \eqref{e:intersect}. So the claim is true and the theorem follows.
\end{proof}

By (2) of Proposition \ref{winproperty}, namely the stability of winning property under countable intersections, one has:
\begin{corollary}\label{coro3}
Let $Y$ be a countable subset of $M$. Then $E(f,Y)$ is winning for modified Schmidt games induced by $f$ and played on $M$.
\end{corollary}

Let $y \in M$, and $\{f_i\}_{i=1}^N$ ($N \geq 2$) be a finite set of $C^{1+\theta}$-expanding endomorphisms on $M$. Since our definition of modified Schmidt games depends on the expanding endomorphism itself, it is nonsense to say whether $\cap_{i=1}^N E(f_i,y)$ is winning for modified Schmidt games. It is not known whether $\cap_{i=1}^N E(f_i,y)$ has full Hausdorff dimension on $M$.

\ \
\\[-2mm]
\textbf{Acknowledgement.} The author would like to thank Federico Rodriguez Hertz, Jinpeng An and Lifan Guan for valuable comments. This research is supported by CPSF (\#2015T80010).


\begin{thebibliography}{10}

\bibitem{AN}
J. An, \emph{Two dimensional badly approximable vectors and Schmidt's game}, Duke Mathematical Journal, to appear.

\bibitem{AL}
C. S. Aravinda and E. Leuzinger, \emph{Bounded geodesics in rank-1 locally symmetric spaces}, Ergodic Theory and Dynamical Systems 15.5 (1995): 813-820.

\bibitem{BFK}
R.~Broderick, L.~Fishman and D.~Y.~Kleinbock, \emph{Schmidt's game, fractals, and orbits of toral endomorphisms}, Ergodic Theory Dynam. Systems 31 (2011), 1095-1107.

\bibitem{BFKRW}
R.~Broderick, L.~Fishman and D.~Y.~Kleinbock, A. Reich, B. Weiss, \emph{The set of badly approximable vectors is strongly $C^1$ incompressible}, Mathematical Proceedings of the Cambridge Philosophical Society. Cambridge University Press, 2012, 153(02): 319-339.

\bibitem{CCM}
J. Chaika, Y. Cheung and H. Masur, \emph{Winning games for bounded geodesics in moduli spaces of quadratic differentials}, Journal of Modern Dynamics 7 (2013), 395--427.

\bibitem{D1}
S. G. Dani, \emph{Divergent trajectories of flows on homogeneous spaces and Diophantine approximation}, J. reine angew. Math 359.55-89 (1985): 102.

\bibitem{D2}
\bysame, \emph{Bounded orbits of flows on homogeneous spaces}, Commentarii Mathematici Helvetici 61.1 (1986): 636-660.

\bibitem{D}
\bysame, \emph{On orbits of endomorphisms of tori and the Schmidt game}, Ergodic Theory and Dynamical Systems 8.04 (1988): 523-529.

\bibitem{DS}
S. G. Dani and H. Shah, \emph{Badly approximable numbers and vectors in Cantor-like sets}, Proceedings of the American Mathematical Society 140.8 (2012): 2575-2587.

\bibitem{Dol}
D. Dolgopyat, \emph{Bounded orbits of Anosov flows}, Duke Mathematical Journal 87.1 (1997): 87-114.

\bibitem{FSU}
L. Fishman, D. S. Simmons and M. Urba\'{n}ski, \emph{Diophantine approximation and the geometry of limit sets in Gromov hyperbolic metric spaces (extended version)}. arXiv preprint arXiv:1301.5630, 2013.

\bibitem{KH}
A. Katok and B. Hasselblatt, \emph{Introduction to the modern theory of dynamical systems}, Vol. 54. Cambridge university press, 1997.

\bibitem{KM}
D.~Y.~Kleinbock and G.~A.~Margulis, \emph{Bounded orbits of nonquasiunipotent flows on homogeneous spaces}, American Mathematical Society Translations (1996): 141-172.

\bibitem{KW1}
D.~Y.~Kleinbock and B.~Weiss, \emph{Modified Schmidt games and Diophantine approximation with weights}, Advances in Mathematics 223.4 (2010): 1276-1298.

\bibitem{KW2}
\bysame, \emph{Modified Schmidt games and a conjecture of Margulis}, Journal of Modern Dynamics, 7 (2013): 429-460.

\bibitem{KW3}
\bysame, \emph{Values of binary quadratic forms at integer points and Schmidt games}, Recent Trends in Ergodic Theory and Dynamical Systems (Vadodara, 2012) (2013): 77-92.

\bibitem{Mc1}
C. McMullen, \emph{Winning sets, quasiconformal maps and Diophantine approximation}, Geometric and Functional Analysis 20.3 (2010): 726-740.

\bibitem{S}
W.~M.~Schmidt, \emph{On badly approximable numbers and certain games}, Transactions of the American Mathematical Society 123.1 (1966): 178-199.

\bibitem{S2}
\bysame, \emph{Badly approximable systems of linear forms}, Journal of Number Theory 1.2 (1969): 139-154.

\bibitem{tseng}
J.~Tseng,  \emph{Schmidt games and Markov partitions}, Nonlinearity, 2009, 22(3): 525-543.

\bibitem{tseng1}
\bysame, \emph{Nondense orbits for Anosov diffeomorphisms of the $2 $-torus}, arXiv preprint arXiv:1503.02273, 2015.

\bibitem{U}
M.~Urba\'{n}ski, \emph{The Hausdorff dimension of the set of points with nondense orbit under a hyperbolic dynamical system}, Nonlinearity 4.2 (1991): 385-397.

\bibitem{Wu}
W. Wu, \emph{Schmidt games and non-dense forward orbits of certain partially hyperbolic systems}, Ergodic Theory
and Dynamical Systems, to appear.

\bibitem{Wu2}
\bysame, \emph{Modified Schmidt games and non-dense forward orbits of  partially hyperbolic systems}, Discrete and Continuous Dynamical Systems, Series A, to appear.
\end{thebibliography}
\end{document}